\documentclass[12pt, a4paper, reqno]{amsart}

\usepackage{a4wide}
\usepackage{enumerate}

\usepackage{color}
\usepackage{ifpdf}
\ifpdf 
  \usepackage[pdftex]{graphicx}
  \DeclareGraphicsExtensions{.pdf,.png,.mps}
  \usepackage{pgf}
  \usepackage{tikz}
\else 
  \usepackage{graphicx}
  \DeclareGraphicsExtensions{.eps,.bmp}
  \usepackage{pgf}
  \usepackage{tikz}
  \usepackage{pstricks}
\fi
\usepackage{epic,bez123}
\usepackage{wrapfig}

\usepackage{amsthm}

\theoremstyle{plain}
\newtheorem{thm}{Theorem}
\newtheorem{cor}[thm]{Corollary}

\newtheorem{lem}[thm]{Lemma}
\newtheorem{prop}[thm]{Proposition}

\theoremstyle{definition}

\theoremstyle{remark}

\usepackage{amsmath}
\DeclareMathOperator{\MD}{MD}
\DeclareMathOperator{\rt}{root}

\newcommand{\abs}[1]{\left|#1\right|}
\newcommand{\set}[1]{\left\{#1\right\}}

\def\T{\mathcal{T}}
\def\F{\mathcal{F}}
\def\G{\mathcal{G}}
\def\H{\mathcal{H}}
\def\A{\mathcal{A}}

\title{On the enumeration of rooted trees with fixed size of maximal decreasing trees}
\author{Seunghyun Seo}
\address[Seunghyun Seo]{Kangwon National University}
\email{shyunseo@kangwon.ac.kr}

\author{Heesung Shin}
\address[Heesung Shin]{Inha University}
\email{shin@inha.ac.kr}

\date{\today}

\begin{document}
\maketitle
\begin{abstract}
Let $\T_{n}$ be the set of rooted labeled trees on $\set{0,\dots,n}$.
A maximal decreasing subtree of a rooted labeled tree is defined by the maximal subtree from the root with all edges being decreasing.
In this paper, we study a new refinement $\T_{n,k}$ of $\T_n$, which is the set of rooted labeled trees whose maximal decreasing subtree has $k+1$ vertices.

\end{abstract}

\section{Introduction}
For a nonnegative integer $n$, let $\T_{n}$ be the set of rooted labeled trees on $[0,n]:=\set{0,\dots,n}$.
For a given rooted labeled tree $T$, a \emph{maximal decreasing subtree} of $T$ is defined by the maximal subtree from the root with all edges being decreasing, denoted by $\MD(T)$. Figure~\ref{fig:tree} illustrates the maximal decreasing subtree of a given tree $T$.
\begin{figure}[t]
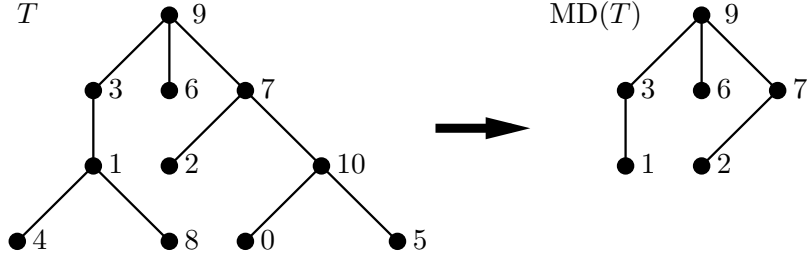

$$
\centering
\begin{pgfpicture}{17.00mm}{16.14mm}{126.00mm}{54.14mm}
\pgfsetxvec{\pgfpoint{1.00mm}{0mm}}
\pgfsetyvec{\pgfpoint{0mm}{1.00mm}}
\color[rgb]{0,0,0}\pgfsetlinewidth{0.30mm}\pgfsetdash{}{0mm}
\pgfcircle[fill]{\pgfxy(40.00,50.00)}{1.00mm}
\pgfcircle[stroke]{\pgfxy(40.00,50.00)}{1.00mm}
\pgfcircle[fill]{\pgfxy(30.00,40.00)}{1.00mm}
\pgfcircle[stroke]{\pgfxy(30.00,40.00)}{1.00mm}
\pgfcircle[fill]{\pgfxy(40.00,40.00)}{1.00mm}
\pgfcircle[stroke]{\pgfxy(40.00,40.00)}{1.00mm}
\pgfcircle[fill]{\pgfxy(50.00,40.00)}{1.00mm}
\pgfcircle[stroke]{\pgfxy(50.00,40.00)}{1.00mm}
\pgfcircle[fill]{\pgfxy(40.00,30.00)}{1.00mm}
\pgfcircle[stroke]{\pgfxy(40.00,30.00)}{1.00mm}
\pgfcircle[fill]{\pgfxy(60.00,30.00)}{1.00mm}
\pgfcircle[stroke]{\pgfxy(60.00,30.00)}{1.00mm}
\pgfcircle[fill]{\pgfxy(30.00,30.00)}{1.00mm}
\pgfcircle[stroke]{\pgfxy(30.00,30.00)}{1.00mm}
\pgfcircle[fill]{\pgfxy(50.00,20.00)}{1.00mm}
\pgfcircle[stroke]{\pgfxy(50.00,20.00)}{1.00mm}
\pgfcircle[fill]{\pgfxy(70.00,20.00)}{1.00mm}
\pgfcircle[stroke]{\pgfxy(70.00,20.00)}{1.00mm}
\pgfcircle[fill]{\pgfxy(40.00,20.00)}{1.00mm}
\pgfcircle[stroke]{\pgfxy(40.00,20.00)}{1.00mm}
\pgfcircle[fill]{\pgfxy(20.00,20.00)}{1.00mm}
\pgfcircle[stroke]{\pgfxy(20.00,20.00)}{1.00mm}
\pgfmoveto{\pgfxy(40.00,50.00)}\pgflineto{\pgfxy(40.00,40.00)}\pgfstroke
\pgfmoveto{\pgfxy(40.00,50.00)}\pgflineto{\pgfxy(50.00,40.00)}\pgfstroke
\pgfmoveto{\pgfxy(50.00,40.00)}\pgflineto{\pgfxy(60.00,30.00)}\pgfstroke
\pgfmoveto{\pgfxy(60.00,30.00)}\pgflineto{\pgfxy(70.00,20.00)}\pgfstroke
\pgfmoveto{\pgfxy(60.00,30.00)}\pgflineto{\pgfxy(50.00,20.00)}\pgfstroke
\pgfmoveto{\pgfxy(50.00,40.00)}\pgflineto{\pgfxy(40.00,30.00)}\pgfstroke
\pgfmoveto{\pgfxy(40.00,50.00)}\pgflineto{\pgfxy(30.00,40.00)}\pgfstroke
\pgfmoveto{\pgfxy(30.00,40.00)}\pgflineto{\pgfxy(30.00,30.00)}\pgfstroke
\pgfmoveto{\pgfxy(30.00,30.00)}\pgflineto{\pgfxy(40.00,20.00)}\pgfstroke
\pgfmoveto{\pgfxy(30.00,30.00)}\pgflineto{\pgfxy(20.00,20.00)}\pgfstroke
\pgfputat{\pgfxy(32.00,39.00)}{\pgfbox[bottom,left]{\fontsize{11.38}{13.66}\selectfont 3}}
\pgfputat{\pgfxy(22.00,19.00)}{\pgfbox[bottom,left]{\fontsize{11.38}{13.66}\selectfont 4}}
\pgfputat{\pgfxy(52.00,19.00)}{\pgfbox[bottom,left]{\fontsize{11.38}{13.66}\selectfont 0}}
\pgfputat{\pgfxy(62.00,29.00)}{\pgfbox[bottom,left]{\fontsize{11.38}{13.66}\selectfont 10}}
\pgfputat{\pgfxy(42.00,19.00)}{\pgfbox[bottom,left]{\fontsize{11.38}{13.66}\selectfont 8}}
\pgfputat{\pgfxy(72.00,19.00)}{\pgfbox[bottom,left]{\fontsize{11.38}{13.66}\selectfont 5}}
\pgfputat{\pgfxy(52.00,39.00)}{\pgfbox[bottom,left]{\fontsize{11.38}{13.66}\selectfont 7}}
\pgfputat{\pgfxy(42.00,29.00)}{\pgfbox[bottom,left]{\fontsize{11.38}{13.66}\selectfont 2}}
\pgfputat{\pgfxy(43.00,49.00)}{\pgfbox[bottom,left]{\fontsize{11.38}{13.66}\selectfont 9}}
\pgfputat{\pgfxy(32.00,29.00)}{\pgfbox[bottom,left]{\fontsize{11.38}{13.66}\selectfont 1}}
\pgfputat{\pgfxy(42.00,39.00)}{\pgfbox[bottom,left]{\fontsize{11.38}{13.66}\selectfont 6}}
\pgfputat{\pgfxy(20.00,49.00)}{\pgfbox[bottom,left]{\fontsize{11.38}{13.66}\selectfont $T$}}
\pgfcircle[fill]{\pgfxy(110.00,50.00)}{1.00mm}
\pgfcircle[stroke]{\pgfxy(110.00,50.00)}{1.00mm}
\pgfcircle[fill]{\pgfxy(100.00,40.00)}{1.00mm}
\pgfcircle[stroke]{\pgfxy(100.00,40.00)}{1.00mm}
\pgfcircle[fill]{\pgfxy(110.00,40.00)}{1.00mm}
\pgfcircle[stroke]{\pgfxy(110.00,40.00)}{1.00mm}
\pgfcircle[fill]{\pgfxy(120.00,40.00)}{1.00mm}
\pgfcircle[stroke]{\pgfxy(120.00,40.00)}{1.00mm}
\pgfcircle[fill]{\pgfxy(110.00,30.00)}{1.00mm}
\pgfcircle[stroke]{\pgfxy(110.00,30.00)}{1.00mm}
\pgfcircle[fill]{\pgfxy(100.00,30.00)}{1.00mm}
\pgfcircle[stroke]{\pgfxy(100.00,30.00)}{1.00mm}
\pgfmoveto{\pgfxy(110.00,50.00)}\pgflineto{\pgfxy(110.00,40.00)}\pgfstroke
\pgfmoveto{\pgfxy(110.00,50.00)}\pgflineto{\pgfxy(120.00,40.00)}\pgfstroke
\pgfmoveto{\pgfxy(120.00,40.00)}\pgflineto{\pgfxy(110.00,30.00)}\pgfstroke
\pgfmoveto{\pgfxy(110.00,50.00)}\pgflineto{\pgfxy(100.00,40.00)}\pgfstroke
\pgfmoveto{\pgfxy(100.00,40.00)}\pgflineto{\pgfxy(100.00,30.00)}\pgfstroke
\pgfputat{\pgfxy(102.00,39.00)}{\pgfbox[bottom,left]{\fontsize{11.38}{13.66}\selectfont 3}}
\pgfputat{\pgfxy(122.00,39.00)}{\pgfbox[bottom,left]{\fontsize{11.38}{13.66}\selectfont 7}}
\pgfputat{\pgfxy(112.00,29.00)}{\pgfbox[bottom,left]{\fontsize{11.38}{13.66}\selectfont 2}}
\pgfputat{\pgfxy(113.00,49.00)}{\pgfbox[bottom,left]{\fontsize{11.38}{13.66}\selectfont 9}}
\pgfputat{\pgfxy(102.00,29.00)}{\pgfbox[bottom,left]{\fontsize{11.38}{13.66}\selectfont 1}}
\pgfputat{\pgfxy(112.00,39.00)}{\pgfbox[bottom,left]{\fontsize{11.38}{13.66}\selectfont 6}}
\pgfputat{\pgfxy(90.00,49.00)}{\pgfbox[bottom,left]{\fontsize{11.38}{13.66}\selectfont $\MD(T)$}}
\pgfsetlinewidth{1.20mm}\pgfmoveto{\pgfxy(75.00,35.00)}\pgflineto{\pgfxy(85.00,35.00)}\pgfstroke
\pgfmoveto{\pgfxy(85.00,35.00)}\pgflineto{\pgfxy(82.20,35.70)}\pgflineto{\pgfxy(82.20,34.30)}\pgflineto{\pgfxy(85.00,35.00)}\pgfclosepath\pgffill
\pgfmoveto{\pgfxy(85.00,35.00)}\pgflineto{\pgfxy(82.20,35.70)}\pgflineto{\pgfxy(82.20,34.30)}\pgflineto{\pgfxy(85.00,35.00)}\pgfclosepath\pgfstroke
\end{pgfpicture}%
$$
\caption{The maximal decreasing subtree of $T$ from the root $9$}
\label{fig:tree}
\end{figure}
Let $\T_{n,k}$ be the set of rooted labeled trees in $\T_n$ whose maximal decreasing subtree has $k+1$ vertices.

Within the scope of proven research, a maximal decreasing subtree first appeared in the paper \cite{CDG00} of Chauve, Dulucq, and Guibert, for constructing the bijection between $\T_{n,0}$ and the set of trees in $\T_n$ with $n$ being a leaf. Recently, Bergeron and Livernet \cite{BL10} introduced it in order to analyze the free Lie algebra based on rooted labeled trees. None of them mentioned, however, the refined set $\T_{n,k}$ nor considered the enumeration of $\T_{n,k}$.


In Section~\ref{sec:main}, we shall count the number of elements in $\T_{n,k}$. We shall also introduce a set of certain functions on $[n]$, which is equinumerous to $\T_{n,k}$.
In Section~\ref{sec:properties}, we shall decompose a rooted labeled tree into rooted subtrees, each maximal decreasing subtree of which is a single vertex. Then some formulae related to $\abs{\T_{n,k}}$ are given from this decomposition.
In Section~\ref{sec:inverse}, using the inverse of the matrix $\left[ {i+j \choose j}\right]_{0 \le i,j \le n}$, $\T_{n,k}$ can be expressed as a linear combination of $\set{(n+1)^n,(n+2)^n,\cdots,(2n+1)^n}$.
In the last section, we discuss bijective proofs of our results.


\section{Main results}\label{sec:main}
First of all, let us count the number of elements in the set $\T_{n,k}$.
\begin{thm}\label{thm:tree}
For nonnegative integers $n$ and $k$, we have
$$\abs{\T_{n,k}} = \sum_{m=k}^n {n+1 \choose m+1} S(m+1,k+1) \, k! \, (n-k)^{n-m-1} (m-k),$$
where $S(n,k)$ is a Stirling number of the second kind.
\end{thm}
\begin{proof}
Given a rooted labeled tree $T$, let $V_1$ be the union of the set of vertices in $\MD(T)$ and the set of children of any vertex in $\MD(T)$. Now, we will count the number of rooted labeled trees $T \in \T_{n,k}$ with $\abs{V_1}=m+1$.


First of all, the number of ways for selecting $V_1$ is equal to $n+1 \choose m+1$. Make a partition of $V_1$ into $k+1$ blocks, namely, $B_1, \dots , B_{k+1}$. The number of such partitions is equal to $S(m+1, k+1)$. Take the set $V_0$ consisting of the minimum $m_i$ of each block $B_i$. Make a decreasing subtree on $V_0 = \set{m_1, \dots, m_{k+1}}$.
Since it is well known that the number of (unordered) increasing trees on $k+1$ nodes is $k!$, 
there are exactly $k!$ ways of making a decreasing subtree on $V_0$.
Append vertices in $V_1\setminus V_0$ to this decreasing subtree such that elements in $B_i \setminus \set{m_i}$ are children of $m_i$ for $i=1,\dots,k+1$.
It is well-known that the number of forests $F$ on $[0,n]\setminus V_0$ such that $V_1\setminus V_0$ is the set of all roots of $F$ is equal to $(n-k)^{n-m-1}(m-k)$ (see \cite[Prop. 5.3.2]{Sta99}).
Since the range of $m$ is $k\le m \le n$,
$$\abs{\T_{n,k}} = \sum_{m=k}^n {n+1 \choose m+1} S(m+1,k+1) \, k! \, (n-k)^{n-m-1} (m-k).$$
\end{proof}

For a nonnegative integer $n$, let $\F_n$ be the set of functions from $[n]$ to $[n]$, where $[n]:=\set{1,\dots,n}$ if $n$ is positive integer and $[0]:=\emptyset$.
Let $\F_{n,k}$ be the set of functions $f \in \F_n$ with $[k] \subset f([n])$, where $f([n])$ is the image of $f$.

\begin{prop}\label{thm:Fn}
For nonnegative integers $n$ and $k$, we have
\begin{align}
\abs{\F_{n,k}} =& \sum_{m=k}^n {n \choose m} S(m,k) \, k! \, (n-k)^{n-m} \label{eq:F1}\\
=& \sum_{i\ge 0} (-1)^i\binom{k}{i}\,(n-i)^n. \label{eq:F2}
\end{align}
\end{prop}

\begin{proof}
Let $f^{-1} ([k]) =A$ and $\abs{A}=m$. The number of  subsets $A$ of $[n]$ of size $m$  is equal to $n \choose m$. The function $f$ can be decomposed into a surjection from $A$ to $[k]$ with $S(m,k)\, k!$ ways and a function from $[n]\setminus A$ to $[n] \setminus [k]$ with $(n-k)^{n-m}$ ways.
Since $m$ runs through from $k$ to $n$, the formula~\eqref{eq:F1} holds.

Meanwhile, defining $A_j$ by the set $\set{f \in \F_n ~|~ f^{-1}(j) = \emptyset}$, we have
$$\F_{n,k} = \F_n \setminus \left( A_1 \cup \dots \cup A_{k} \right).$$
By the principle of inclusion and exclusion, we have
\begin{align*}
\abs{\F_{n,k}} &= \abs{\F_n} - \abs{A_1 \cup \dots \cup A_{k}} = \sum_{I\subset [k]} (-1)^{\abs{I}} (n-\abs{I})^n.
\end{align*}
So, the formula \eqref{eq:F2} holds.
\end{proof}

\begin{thm}\label{thm:eq}
For nonnegative integers $n$ and $k$, we have $\abs{\T_{n,k}} = \abs{\F_{n,k}}$, i.e.,
\begin{equation}\label{eq:T=F}
\sum_{m=k}^n {n+1 \choose m+1} S(m+1,k+1) \, k! \, (n-k)^{n-m-1} (m-k) = \sum_{i=0}^k (-1)^{i}\binom{k}{i}\,(n-i)^n.
\end{equation}
\end{thm}

\begin{proof}
Since $S(m+1, k+1) (k+1)!$ is the number of surjective functions from $[m+1]$ to $[k+1]$, which is equal to $\sum_{j\ge0}(-1)^j\binom{k+1}{j}(k+1-j)^{m+1}$ by the principle of inclusion and exclusion, it follows that
\begin{align*}
\abs{\T_{n,k}}
=&\sum_{m} {n+1 \choose m+1} \left[\,S(m+1,k+1) \, k! \,\right]\, (n-k)^{n-m-1} (m-k)\\
=&\sum_{m} {n+1 \choose m+1} \left[\frac{1}{k+1}\sum_{j\ge0}(-1)^j\binom{k+1}{j}(k+1-j)^{m+1} \right]\, (n-k)^{n-m-1} (m-k).
\end{align*}
Separating the term $m-k$ to $(n-k)-(n-m)$ and changing the order of summations, we get
\begin{align*}
\abs{\T_{n,k}}
=&\sum_{j} \frac{(-1)^j}{k+1}\binom{k+1}{j}\sum_{m} \binom{n+1}{m+1}(k+1-j)^{m+1}(n-k)^{n-m}\\
& ~-~\sum_{j} \frac{(-1)^j}{k+1}\binom{k+1}{j}\sum_{m} (n+1)\binom{n}{m+1}(k+1-j)^{m+1}(n-k)^{n-m-1}.
\end{align*}
By the binomial theorem,
\begin{align*}
\abs{\T_{n,k}}
=&\sum_{j} \frac{(-1)^j}{k+1}\binom{k+1}{j}\left[\,(n+1-j)^{n+1} - (n+1)(n+1-j)^n \,\right]\\
=&\sum_{j} \frac{(-1)^j}{k+1}\binom{k+1}{j}\,(-j)(n+1-j)^n.
\end{align*}
Substituting $j=i+1$ in  the previous equation, the formula \eqref{eq:T=F} holds.


\end{proof}

\section{Properties} \label{sec:properties}

A rooted labeled tree $T$ is called a \emph{local minimum tree}, if $\MD(T)$ consists of a single vertex. Note that $\T_{n,0}$ is the set of local minimum trees on $[0,n]$ and  $\abs{\T_{n,0}}$ is equal to $n^n$ \cite{CDG00}. Also $\T_{n,n}$ is the set of decreasing trees on $[0,n]$ and $\abs{\T_{n,n}}$ is equal to $n!$.

Given $T \in \T_{n,k}$, we can decompose $T$ into $k+1$ local minimum trees by removing $k$ edges in $\MD(T)$. This decomposition yields the following lemma.

\begin{lem}\label{lemma:decompose}
For nonnegative integers $n$ and $k$, we have
\[
\abs{\T_{n,k}} =
\frac{1}{k+1} \sum_{n_1+\dots + n_{k+1}=n-k} {n+1 \choose n_1+1, \dots , n_{k+1}+1} {n_1}^{n_1} \cdots {n_{k+1}}^{n_{k+1}}
\]
\end{lem}
\begin{proof}
It is enough to show the following formula
\begin{equation}\label{eq:decompose}
\abs{\T_{n,k}} =
k! \left( \frac{1}{(k+1)!} \sum_{n_1+\dots + n_{k+1}=n-k} {n+1 \choose n_1+1, \dots , n_{k+1}+1} \abs{\T_{n_1,0}} \cdots \abs{\T_{n_{k+1},0}} \right).
\end{equation}

First of all, we will make a tuple $(T_1,\dots, T_{k+1})$ of local minimum trees satisfying two conditions:
\begin{enumerate}[(i)]
\item \label{item:labelpartition}
The tuple $(L(T_1),\dots, L(T_{k+1}))$ is an ordered partition of the set $[0,n]$ and
\item \label{item:ordering}
$\rt(T_1) < \rt(T_1) < \dots < \rt(T_{k+1}),$
\end{enumerate}
where $L(T)$ means a set of labels of vertices in $T$ and $\rt(T)$ a label of the root of $T$.

Consider a tuple $(S_1,\dots, S_{k+1})$ of local minimum trees with the only condition \eqref{item:labelpartition}. For a given sequence $n_1, \dots, n_{k+1}$ of nonnegative integers with $n_1+\dots + n_{k+1}=n-k,$
the number of tuples $(S_1,\dots, S_{k+1})$ with $\abs{L(S_i)}=n_i+1$ is equal to
$${n+1 \choose n_1+1, \dots , n_{k+1}+1} \abs{\T_{n_1,0}} \cdots \abs{\T_{n_{k+1},0}}.$$
So the number of all tuples $(S_1,\dots, S_{k+1})$ with the condition \eqref{item:labelpartition} is equal to
$$\sum_{n_1+\dots + n_{k+1}=n-k} {n+1 \choose n_1+1, \dots , n_{k+1}+1} \abs{\T_{n_1,0}} \cdots \abs{\T_{n_{k+1},0}}.$$
From the condition \eqref{item:ordering}, the number of all tuples $(T_1,\dots, T_{k+1})$ is equal to
$$\frac{1}{(k+1)!} \sum_{n_1+\dots + n_{k+1}=n-k} {n+1 \choose n_1+1, \dots , n_{k+1}+1} \abs{\T_{n_1,0}} \cdots \abs{\T_{n_{k+1},0}}.$$
Since the number of decreasing subtrees on $\set{\rt(T_1),\dots,\rt(T_{k+1})}$ is $k!$, we get the formula~\eqref{eq:decompose}.
\end{proof}

From Lemma~\ref{lemma:decompose}, we can deduce the following result.
\begin{thm}
We have three exponential generating functions:
\begin{align}
1+ \sum_{n\ge 0} \sum_{k=0}^n \abs{\T_{n,k}} \frac{t^{k+1}}{k!} \frac{x^{n+1}}{(n+1)!} &= \exp\left( t\, \sum_{n\ge 0} n^n\frac{x^{n+1}}{(n+1)!} \right), \label{eq:gen1}\\
1+ \sum_{n\ge 0} \sum_{k=0}^n \abs{\T_{n,k}} (k+1) t^{k+1} \frac{x^{n+1}}{(n+1)!} &= \left(1-t\, \sum_{n\ge 0} n^n\frac{x^{n+1}}{(n+1)!}\right)^{-1}, \label{eq:gen2} \\
\sum_{n\ge 0} \sum_{k=0}^n \abs{\T_{n,k}} t^{k+1} \frac{x^{n+1}}{(n+1)!} &= - \ln\left(1-t\, \sum_{n\ge 0} n^n\frac{x^{n+1}}{(n+1)!}\right). \label{eq:gen3}
\end{align}
\end{thm}

\begin{proof}
From Lemma~\ref{lemma:decompose}, left-hand side of three formulas become
\begin{align}
&1+ \sum_{n\ge 0} \sum_{k=0}^n \sum_{n_i} {n+1 \choose n_1+1, \dots , n_{k+1}+1} {n_1}^{n_1} \cdots {n_{k+1}}^{n_{k+1}} \frac{t^{k+1}}{(k+1)!} \frac{x^{n+1}}{(n+1)!}, \label{eq:change1}\\
&1+ \sum_{n\ge 0} \sum_{k=0}^n \sum_{n_i} {n+1 \choose n_1+1, \dots , n_{k+1}+1} {n_1}^{n_1} \cdots {n_{k+1}}^{n_{k+1}} t^{k+1} \frac{x^{n+1}}{(n+1)!}, \label{eq:change2}\\
&\sum_{n\ge 0} \sum_{k=0}^n \sum_{n_i} {n+1 \choose n_1+1, \dots , n_{k+1}+1} {n_1}^{n_1} \cdots {n_{k+1}}^{n_{k+1}} \frac{t^{k+1}}{k+1} \frac{x^{n+1}}{(n+1)!}, \label{eq:change3}
\end{align}
where $n_i$ means $n_1+\dots + n_{k+1}=n-k$.
Using the compositional formula for exponential structures \cite[Theorem 5.5.4]{Sta99},
three formulas are of form $F(t\,G(x))$ where $$G(x)=\sum_{n\ge 0} n^n \frac{x^{n+1}}{(n+1)!}.$$
In case \eqref{eq:change1}, the corresponding $F(x)$ is given by
$$F(x)= 1+\sum_{i\ge 0} \frac{(i+1)!}{(i+1)!} \, \frac{x^{i+1}}{(i+1)!} = \exp(x).$$
In case \eqref{eq:change2}, the corresponding $F(x)$ is given by
$$F(x)= 1+\sum_{i\ge 0} \frac{(i+1)!}{1} \, \frac{x^{i+1}}{(i+1)!} = \frac{1}{1-x}.$$
In case \eqref{eq:change3}, the corresponding $F(x)$ is given by
$$F(x)= \sum_{i \ge 0} \frac{(i+1)!}{(i+1)} \, \frac{x^{i+1}}{(i+1)!} =\ln \frac{1}{1-x}. $$
These complete the proof.
\end{proof}

By definition of $\T_{n,k}$, we have
\begin{equation}\label{eq:def}
\sum_{k=0}^n \abs{\T_{n,k}} = (n+1)^n,
\end{equation}
which can be also induced from $t=1$ in \eqref{eq:gen3}.
Similarly, putting $t=1$ in \eqref{eq:gen2} and applying the equation (5.67) in \cite{Sta99}, we get
\begin{align*}
1+ \sum_{n\ge 0} \sum_{k=0}^n \abs{\T_{n,k}} (k+1) \frac{x^{n+1}}{(n+1)!}
&= \left(1- \sum_{n\ge 0} n^n \frac{x^{n+1}}{(n+1)!}\right)^{-1} \\
&= 1+ \sum_{n\ge 0} (n+2)^n \frac{x^{n+1}}{(n+1)!}.
\end{align*}
Thus we have
\begin{equation}\label{eq:simple}
\sum_{k=0}^n (k+1) \abs{\T_{n,k}} = (n+2)^n.
\end{equation}
From \eqref{eq:def} and \eqref{eq:simple}, we are able to deduce the followings.

\begin{thm}\label{thm:T}
For a nonnegative integers $n$, $k$, and $\alpha$, we have
\begin{equation}\label{eq:recur}
\sum_{k=0}^n {k+\alpha \choose \alpha} \abs{\T_{n,k}} = (n+1+\alpha)^n.
\end{equation}
\end{thm}

\begin{proof}

Since we have proved $\abs{\T_{n,k}} = \abs{\F_{n,k}}$ in Theorem~\ref{thm:eq}, it is enough to show
\begin{equation}\label{eq:recurT}
\sum_{k=0}^n {k+\alpha \choose \alpha} \abs{\F_{n,k}} = (n+1+\alpha)^n.
\end{equation}

For $\alpha=0$,
let $\G_{n,k}$ be the set of functions $g$ from $[n]$ to $[0,n]$ with $[0, k-1] \subset g([n])$ but $k \not \in g([n])$.
By definition of $\G_{n,k}$,
\[
\sum_{k=0}^n \abs{\G_{n,k}} = (n+1)^n.
\]
There is a simple bijection $\varphi$ from $\F_{n,k}$ to $\G_{n,k}$ as follows:
Given a $f \in \F_{n,k}$, consider a function $g$ from $[n]$ to $[0,n]$ defined by
\[
g(i) =
\begin{cases}
f(i)-1 & \text{if $f(i)\le k$}\\
f(i) & \text{otherwise.}
\end{cases}
\]
Since the images of $g$ includes $0, \dots, k-1$ but does not include $k$, the function $g$ belongs to $\G_{n,k}$ and $\varphi(f)=g$ is well-defined.
Since $\varphi$ is reversible, it is a bijection. So it holds that
\[
\abs{\F_{n,k}} = \abs{\G_{n,k}}
\]
for all $0\le k \le n$.

For $\alpha>0$, let $\H_{n,k,\alpha}$ be the set of functions $g$ from $[n]$ to $[-\alpha,n]$ with $$\abs{[-\alpha, k-1] \setminus g([n])}=\alpha$$ and $k \not \in g([n])$, where $[a,b]:=\set{a,\dots,b}$.
For every function $h$ from $[n]$ to $[-\alpha, n]$,
since the cardinality of the domain is less than the cardinality of the codomain by $\alpha+1$,  there exists a unique $k$ satisfying above conditions. Note that $k$ is the $(\alpha+1)$-st element in $[-\alpha,n] \setminus g([n])$. Thus
\[
\sum_{k=0}^n \abs{\H_{n,k,\alpha}} = (n+1+\alpha)^n.
\]\
Let $\A$ be the set of all $\alpha$-elements subsets of $[-\alpha, k-1]$.
Clearly, $\abs{\A} = {k+\alpha \choose \alpha}$.
There is a bijection from $\A \times \F_{n,k}$ to $\H_{n,k,\alpha}$ as follows:
For a given $(A,f) \in \A \times \F_{n,k}$, we make a $(A,\varphi(f)) \in \A \times \G_{n,k}$.
Consider the order-preserving bijection $\sigma$ from $[0,n]$ to $[-\alpha, n]\setminus A$. Then we can define the function $h$ from $[n]$ to $[-\alpha,n]$ by
$$h = \sigma\circ(\varphi(f))$$
and this function $h$ is contained in  $\H_{n,k,\alpha}$.
Hence,
\[
\abs{\H_{n,k,\alpha}} = \abs{\A\times \F_{n,k}} = \abs{\A}\cdot\abs{\F_{n,k}} = {k+\alpha \choose \alpha}\abs{\F_{n,k}}
\]
for all $0\le k \le n$.
\end{proof}

For example, let $n=5$, $k=2$, and $\alpha=3$.
Take $A = \set{-2,-1,1} \in {[-3,1] \choose 3}$ and $f = (f(1), \dots, f(5)) = (5,2,1,3,2) \in \F_{5,2}$.
Then $g=\varphi(f)  \in \G_{5,2}$ and $h=\sigma\circ(\varphi(f)) \in \H_{5,2,3}$ are given by
\begin{align*}
(g(1),\dots, g(5)) &= (5,1,0,3,1),\\
(h(1),\dots, h(5)) &= (5,0,-3,3,0).
\end{align*}

Let us consider the equation \eqref{eq:recur} or \eqref{eq:recurT} for a negative integer $\alpha$. In fact, the left hand sides of these equations are not well-defined even for $\alpha=-1$, nevertheless the right hand sides are. Here we find, however, the coefficients of $\abs{\T_{n,k}}$ that can replace the term $k-1 \choose -1$ as follows.

\begin{thm}\label{thm:remark}
For a positive integer $n$, we have
\begin{equation}\label{eq:div}
\sum_{k=1}^n \frac{1}{k}\abs{\T_{n,k}} = \sum_{k=1}^n \frac{1}{k}\abs{\F_{n,k}} =n^n.
\end{equation}
\end{thm}

\begin{proof}
From Proposition~\ref{thm:Fn} and Theorem~\ref{thm:eq}, expanding $(n-i)^n$ by the binomial theorem, we have
\begin{align*}
\abs{\T_{n,k}} = \abs{\F_{n,k}} &= \sum_{i\ge 0} (-1)^i\binom{k}{i}\,(n-i)^n\\
&= \sum_{i\ge 0} (-1)^i\binom{k}{i} \sum_{j \ge 0} \binom{n}{j} n^{n-j} (-i)^j.
\end{align*}
Hence the left hand side of \eqref{eq:div} is
$$
\sum_{k=1}^n \frac{1}{k}\abs{\T_{n,k}}
= \sum_{j \ge 0} \sum_{k=1}^n \sum_{i\ge 0}  \frac{1}{k} (-1)^i\binom{k}{i} \binom{n}{j} n^{n-j} (-i)^j.
$$
We divide it into three cases; $j=0$, $j=1$, and $j>1$.

In case of $j=0$,
\begin{equation}\label{eq:case1}
\sum_{k=1}^n \sum_{i\ge 0} \frac{1}{k} (-1)^i\binom{k}{i} n^{n} = n^n \sum_{k=1}^n \frac{1}{k} \,(1-1)^k =0.
\end{equation}

In case of $j=1$,
\begin{align}
\sum_{k=1}^n \sum_{i\ge 0}  \frac{1}{k} (-1)^i\binom{k}{i} n^{n} (-i)
&= n^n \sum_{k=1}^n \sum_{i\ge 0} (-1)^{i-1}\binom{k-1}{i-1} \notag \\
&= n^n \sum_{k=1}^n (1-1)^{k-1} = n^n. \label{eq:case2}
\end{align}

In case of $j>1$,
\begin{align}
\sum_{j > 1} \sum_{k=1}^n \sum_{i\ge 0} \frac{1}{k} (-1)^i &\binom{k}{i} \binom{n}{j} n^{n-j} (-i)^j \notag\\
&= \sum_{j > 1} \sum_{i\ge 0}  (-1)^{i} \binom{n}{j} n^{n-j}\, (-1)^{j} \, i^{j-1} \sum_{k=1}^n \binom{k-1}{i-1} \notag \\
&= \sum_{j > 1} \sum_{i\ge 0}  (-1)^{i+j} \binom{n}{j} n^{n-j} i^{j-1} \binom{n}{i} \notag\\
&= \sum_{j > 1} (-1)^{j} \binom{n}{j} n^{n-j} \left[ \sum_{i\ge 0}  (-1)^{i} \binom{n}{i} i^{j-1}\right] =0. \label{eq:case3}
\end{align}

Note that, by the principle of inclusion and exclusion, the expression $\sum_{i\ge 0}  (-1)^{i} \binom{n}{i} i^{j-1}$ in \eqref{eq:case3} is the number of surjections from $[j-1]$ to $[n]$. Since $j-1<n$, it is zero.

From \eqref{eq:case1}, \eqref{eq:case2}, and \eqref{eq:case3}, we finally obtain
$\sum_{k=1}^n \frac{1}{k}\abs{\T_{n,k}}= 0 + n^n + 0 = n^n.$
\end{proof}

\section{Inverse Relation}\label{sec:inverse}
For nonnegative integer $n$, let $A(n)$ be the square matrix of size $n+1$ defined by
$$
A(n) := \left[\binom{i+j}{i}\right]_{0 \le i,j \le n}.
$$
Define two column vectors $t(n)$ and $p(n)$ by
\[
t(n):=\left(
        \begin{array}{c}
          \abs{\T_{n,0}} \\
          \abs{\T_{n,1}} \\
          \vdots \\
          \abs{\T_{n,n}} \\
        \end{array}
      \right)
\quad\text{and}\quad
p(n):=\left(
        \begin{array}{c}
          (n+1)^n \\
          (n+2)^n \\
          \vdots \\
          (2n+1)^n \\
        \end{array}
      \right).
\]
Then the equation~\eqref{eq:recur} can be interpreted as
$$
A(n) \,t(n) = p(n) .
$$
The matrix $A(n)$ is nonsingular. Moreover, we can compute its inverse directly.
Let $B(n)$ be the square matrix of size $n+1$ defined by
$$
B(n) := \left[\,(-1)^{i+j}\sum_{m=0}^n \binom{m}{i}\binom{m}{j}\,\right]_{0 \le i,j \le n}.
$$
\begin{thm}\label{lem:inv}
For a nonnegative integer $n$, the two matrices $A(n)$ and $B(n)$ are inverse matrices of each other.
\end{thm}
\begin{proof}
Here is the calculation:
\begin{align*}
\sum_{l=0}^{n} A(n)_{i,l} B(n)_{l,j}
&=\sum_{l=0}^{n} {i+l \choose i} (-1)^{l+j}\sum_{m=0}^n \binom{m}{l}\binom{m}{j}  \\
&=\sum_{m=0}^n (-1)^{m-j} {m \choose j} \left[ \sum_{l=0}^n {i+l \choose l} {m \choose m-l} (-1)^{m-l} \right]
\end{align*}
Comparing the coefficients of $q^m$ for
$(1-q)^{-(i+1)} (1-q)^m = (1-q)^{-(i+1-m)}$,
we obtain $$\sum_{l=0}^n {i+l \choose l} {m \choose m-l} (-1)^{m-l} = {i \choose m}.$$
Thus, we have 
\begin{align*}
\sum_{l=0}^{n} A(n)_{i,l} B(n)_{l,j}
&=\sum_{m=0}^n (-1)^{m-j} {m \choose j}  {i \choose m}\\
&= {i \choose j} \sum_{m=0}^n (-1)^{m-j}   {i-j \choose m-j} {i \choose j} (1-1)^{i-j} = \delta_{i,j}
\end{align*}
which completes the proof.
\end{proof}

From the matrix identity $B(n) \, p(n) = t(n)$, we obtain another expression for $\abs{\T_{n,k}}$.
\begin{cor} For nonnegative integers $n$ and $k$, we have
\begin{equation}\label{eq:another}
\abs{\T_{n,k}} = \sum_{0 \le l \le m \le n} (-1)^{k+l}\binom{m}{k} \binom{m}{l} (n+1+l)^n\,.
\end{equation}
\end{cor}

\section{Remarks}\label{sec:conj}
Since $\abs{\T_{n,k}} = \abs{\F_{n,k}}$,
it is desired to construct a bijection between $\T_{n,k}$ and $\F_{n,k}$ for all $0 \le k \le n$.
Also, it is natural to ask a bijective proof of \eqref{eq:div}.
Recently, Jang Soo Kim \cite{Kim11} constructed bijections for the above questions.
It would be interesting to give a combinatorial explanation of \eqref{eq:another}.

\section*{Acknowleagement}
The authors thank to  Fr\'ed\'eric Chapoton who first asked this question while the second author in Lyon, and to the anonymous referees for their valuable comments and suggestions to improve this paper.
This study was supported by 2008 Research Grant from Kangwon National University to the first author. 
For the second author, this work was supported by INHA UNIVERSITY Research Grant (INHA-42830).




\end{document}